\documentclass[a4paper,leqno,twoside,11pt]{amsart}
\usepackage{amsmath,amsthm, amssymb, amscd, verbatim}
\usepackage{graphicx,color}

\RequirePackage{hyperref}
\RequirePackage{graphicx}

\date{\today}

\newcommand\thistitle{Wiring Switches to Light Bulbs}
\newcommand\thistitleshort{Wiring Switches to Light Bulbs}
\title[\thistitleshort]{\thistitle}

\newcommand\thisauthor{Stephen M. Buckley and Anthony G. O'Farrell}
\newcommand\thisauthorshort{Buckley and O'Farrell}
\author[\thisauthorshort]{\thisauthor}

\address[Both authors]{%
Department of Mathematics and Statistics, Maynooth University,
Maynooth, Co.~Kildare, W23 HW31, Ireland}

\email{stephen.m.buckley@mu.ie, anthony.ofarrell@mu.ie}

\thanks{The first author was partly supported by Science Foundation Ireland.
Both authors were partly supported by the European Science Foundation
Networking Programme HCAA}

\keywords{wiring, switching, MAX-XOR-SAT, Hamming distance}
\subjclass[2020]{Primary: 05D99. Secondary: 11B39, 68R05, 94C10}

\newcommand\ignore[1]{}

\def\la{\lambda}
\def\N{\mathbb{N}}              \def\Z{\mathbb{Z}}
\def\F{\mathbb{F}}
\def\M{\mathcal{M}}
\def\({\left(}           
\def\){\right)}          
\def\lce{\left\lceil}    
\def\rce{\right\rceil}

\def\deg{\operatorname{deg}}  \def\diag{\operatorname{diag}}

\def\ds{\displaystyle}

\numberwithin{equation}{subsection}
\theoremstyle{plain}
\newtheorem{thm}{Theorem}[section]    
\newtheorem{lem}[thm]{Lemma}      \newtheorem{cor}[thm]{Corollary}

\theoremstyle{remark}

\newtheorem{obs}[thm]{Observation}

\catcode`\@=11       

\def\rf#1{\@rf{#1}#1:;;}
\def\rfs#1{\@rfs{#1}#1:;;}
\def\rfm#1{\@rfF#1<>;;}

\def\@C{C}\def\@E{E}\def\@F{F}\def\@L{L}\def\@O{O}\def\@P{P} 
\def\@Q{Q}\def\@R{R}\def\@S{S}\def\@T{T}\def\@X{X}\def\@s{s} 

\def\@rf#1#2:#3;;{\xdef\@b{#2}
  \ifx\@b\@C Corollary~\ref{#1}\else%
  \ifx\@b\@E (\ref{#1})\else
  \ifx\@b\@F Fact~\ref{#1}\else%
  \ifx\@b\@L Lemma~\ref{#1}\else%
  \ifx\@b\@O Observation~\ref{#1}\else%
  \ifx\@b\@P Proposition~\ref{#1}\else%
  \ifx\@b\@Q Question~\ref{#1}\else%
  \ifx\@b\@R Remark~\ref{#1}\else%
  \ifx\@b\@S Section~\ref{#1}\else%
  \ifx\@b\@T Theorem~\ref{#1}\else%
  \ifx\@b\@X Example~\ref{#1}\else%
  \ifx\@b\@s \S\ref{#1}\else
  \ref{#1}\fi\fi\fi\fi\fi\fi\fi\fi\fi\fi\fi\fi}

\def\@rfs#1#2:#3;;{\def\@b{#2}
  \ifx\@b\@C Corollaries~\ref{#1}\else%
  \ifx\@b\@F Facts~\ref{#1}\else%
  \ifx\@b\@L Lemmas~\ref{#1}\else%
  \ifx\@b\@O Observations~\ref{#1}\else%
  \ifx\@b\@P Propositions~\ref{#1}\else%
  \ifx\@b\@Q Questions~\ref{#1}\else%
  \ifx\@b\@R Remarks~\ref{#1}\else%
  \ifx\@b\@S Sections~\ref{#1}\else%
  \ifx\@b\@T Theorems~\ref{#1}\else%
  \ifx\@b\@X Examples~\ref{#1}\else%
  \ifx\@b\@D Definitions~\ref{#1}\else
  \ref{#1}\fi\fi\fi\fi\fi\fi\fi\fi\fi\fi\fi}

\def\@rfF<#1>#2;;{\def\@c{#2}
  \@rfs{#1}#1:;;\ifx\@c\empty\else\@rfL:#2;;\fi}

\def\@rfL:#1<#2>#3;;{\def\@b{#2}\def\@c{#3}
  #1\ifx\@b\empty\else\ref{#2}\ifx\@c\empty\else\@rfL:#3;;\fi\fi}

\catcode`\@=12       

\newcommand\authorbio{
{\bf Stephen Buckley} MRIA studied at UCC and Chicago, and worked at 
the University of Michigan before coming to Maynooth.
He has been Head of the Department of Mathematics and Statistics
since 2007. More at \url{https://archive.maths.nuim.ie/staff/sbuckley/}.
\\
{\bf Anthony G. O'Farrell} MRIA studied at UCD and Brown, and worked at UCLA
before taking the chair of Mathematics in Maynooth, where he served for 
37 years and continues as Professor Emeritus. 
More at \url{https://www.logicpress.ie/aof}.
}

\begin{document}

\begin{abstract}
Given $n$ buttons and $n$ bulbs so that the $i$th button toggles the $i$th
bulb and at most two other bulbs, we compute the sharp lower bound on the
number of bulbs that can be lit regardless of the action of the buttons.
\end{abstract}
\maketitle

\section{Introduction}

\subsection{Origins}
The following problem was posed in the 2008 Irish Intervarsity Mathematics
Competition\footnote{Set by the first author. One proof he gave at that time
established Theorem 1.1(b) below by exploiting the discrete dynamical 
systems associated to $\mu^*(n,2)$ in a manner similar to the proof in 
Subsection 4.1.}:
\medskip

\begin{quote}
In a room there are 2008 bulbs and 2008 buttons, both sets numbered from 1
to 2008.  For $1\le i\le 2008$, pressing Button $i$ changes the on/off
status of Bulb $i$ and one other bulb (the same other bulb each
time). Assuming that all bulbs are initially off, prove that by pressing the
appropriate combination of buttons we can simultaneously light at least 1340
of them. Prove also that in the previous statement, 1340 cannot be replaced
by any larger number.
\end{quote}

\medskip

This problem, henceforth referred to as the {\it Prototype Problem}, can be
generalized in a variety of ways:

\begin{enumerate}
\item Most obviously, ``2008'' can be replaced by a general integer $n$.
\item We can consider more general wirings $W$, where each button
    switches the on/off status of a (possibly non-constant) number of
    bulbs.
\item We may consider initial configurations $c$ where not all of the
    bulbs are off.
\item We however insist that the numbers of buttons and bulbs are equal,
    and that Button $i$ changes the on/off status of Bulb $i$, $1\le
    i\le n$.
\end{enumerate}

Figure \ref{Figure1} is a sketch of a typical wiring.
\begin{figure}[ht]
\begin{center}
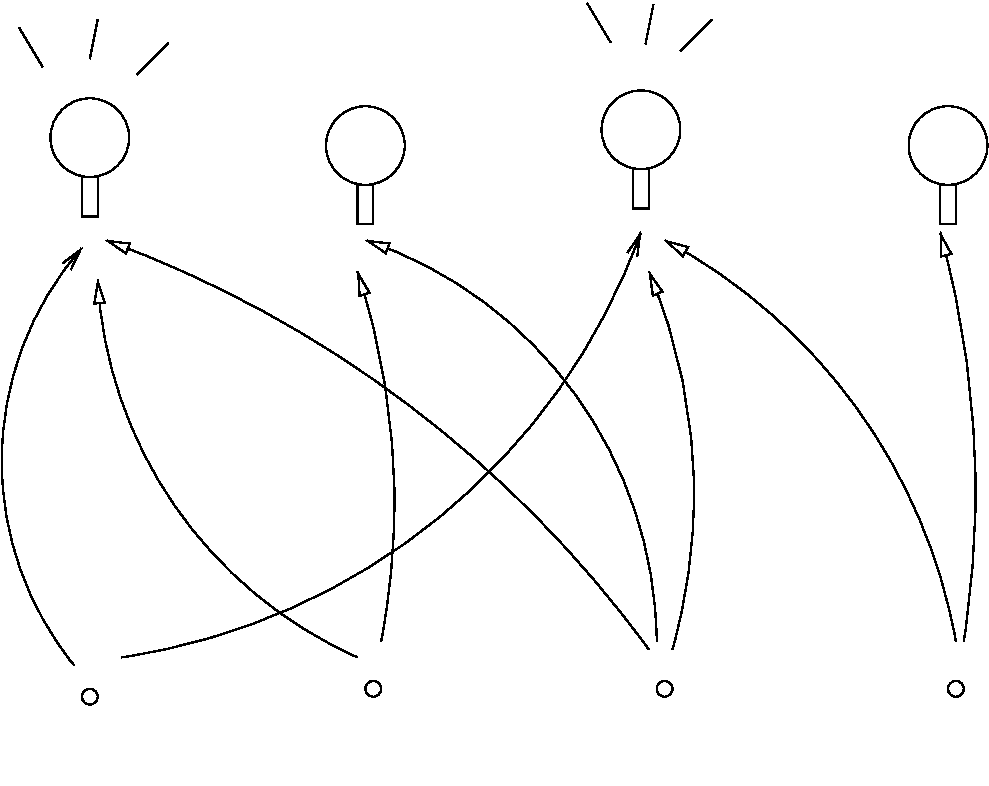
\caption{A Wiring}\label{Figure1}
\end{center}
\end{figure}

These problems are related to the type of problem known as
MAX-XOR-SAT in Computer Science. We discuss this connection in more detail
in Subsection \ref{SS:MXS} below. There may also be a connection to
a meta-Fibonacci sequence related to A046699. See \cite{BOF}.

\subsection{Notation}
Before we continue, let us introduce a little notation. For a fixed wiring
$W$, where the initial on/off configuration of the bulbs is given by $c$, let
$M(W,c)$ be the maximum number of bulbs that can be lit by pressing any
combination of the buttons.

Suppose $n,m\ge 1$. Let $\mu(n,m)$ be the minimum value of $M(W,c)$ over all
wirings $W$ of $n$ buttons and bulbs, where Button $i$ is connected to {\it
at most} $m$ bulbs, including Bulb $i$, for each $1\le i\le n$, and
initially all bulbs are off (which we write as ``$c=0$''). If additionally
$n\ge m$, let $\mu^*(n,m)$ be the minimum value of $M(W,c)$ over all wirings
$W$ of $n$ buttons and bulbs, where Button $i$ is connected to {\it exactly}
$m$ bulbs, including Bulb $i$, for each $1\le i\le n$, and $c=0$. Thus the
Prototype Problem is to show that $\mu^*(2008,2)=1340$.

We define $\mu(n)=\mu(n,n)$, which trivially equals $\mu(n,m)$ for all
$m>n$. Thus $\mu(n)$ is the minimum value of $M(W,0)$, over all wirings of
the $n$ buttons, subject only to condition (d) above.

We also define $\nu(n,m)$, $\nu^*(n,m)$, and $\nu(n)$ in a similar manner to
$\mu(n,m)$, $\mu^*(n,m)$, and $\mu(n)$, respectively, except that we take
the minima over all possible initial configurations $c$, rather than taking
$c=0$. In this article, we are mainly interested in $\mu(n,m)$ and
$\mu^*(n,m)$, and we compute these functions for $m\le 3$. However 
the more
easily calculated $\nu$-variants provide very useful explicit lower bounds
(cf. Theorem \ref{T:nu} below).

\subsection{Results}
Our first theorem gives formulae for $\mu(n,2)$ and $\mu^*(n,2)$; note that
$\mu(n,2)=\mu^*(n,2)$ except when $n\equiv 1 \mod 3$.

\begin{thm}\label{T:m=2} Let $n\in\N$.
\begin{enumerate}
\item $\mu(n,2) = \lce\ds{2n/3}\rce$.
\item If $n\ge 2$, then $\mu^*(n,2) = 2\lce\ds{n/3}\rce$ is the
    least even integer not less than $\mu(n,2)$.
\end{enumerate}
\end{thm}

Next we give formulae for $\mu(n,3)$ and $\mu^*(n,3)$.
\begin{thm}\label{T:m=3} Let $n\in\N$.
\begin{enumerate}
\item $\mu(n,3) = \mu(n,2)$.
\item If $n\ge 3$, then
$$
\mu^*(n,3) =
  \begin{cases}
  4k-1,     & n = 6k-3 \text{ for some } k\in\N, \\
  \mu(n,3), & \text{otherwise}.
  \end{cases}
$$
\end{enumerate}
\end{thm}
Note that $\mu^*(n,3)
= \mu(n,3)+1$ in the exceptional case $n=6k-3$.

We shall discuss $\mu(n,m)$ and $\mu^*(n,m)$ in the case $m>3$,
(and the relationship to a meta-Fibonacci sequence)
in a separate article
\cite{BOF}. Let us simply note here only that $\mu(n,m)$ and $\mu^*(n,m)$ are no
longer asymptotic to $2n/3$ for large $n$, when $m\ge4$. For instance, we
prove in \cite{BOF} that $\mu(n,4)$ is asymptotic to $4n/7$, and that
$\liminf\limits_{n\to\infty} \mu(n)/n=1/2$. 
\smallskip

After some preliminaries in the next section, we prove general formulae
for $\nu(n,m)$ and $\nu^*(n,m)$ in \rf{S:nu}. We then prove \rf{T:m=2} in
\rf{S:m=2} and \rf{T:m=3} in \rf{S:m=3}.

We wish to thank David Malone for pointing out the connection between our
results and SAT.  We are grateful to the referee for some comments that improved the exposition.

\section{Notation and terminology} \label{S:notation}

\subsection{Graphs}
The notation and terminology introduced in this section will be used
throughout the article. We begin by recasting our problem. First note that we
can replace the twin notions of buttons and bulbs with the single notion of
vertices: when a vertex is pressed, the on/off state of that vertex and some
other vertices is switched. The vertex set $S:=S(n):=\{1,\dots,n\}$ is associated
with a directed graph $G$: we draw an edge from vertex $i$ to each vertex
whose on/off status is altered by pressing vertex $i$.
Figure \ref{Figure2} shows a representation of the
directed graph corresponding to the wiring in
Figure \ref{Figure1}.
\begin{figure}[ht]
\begin{center}
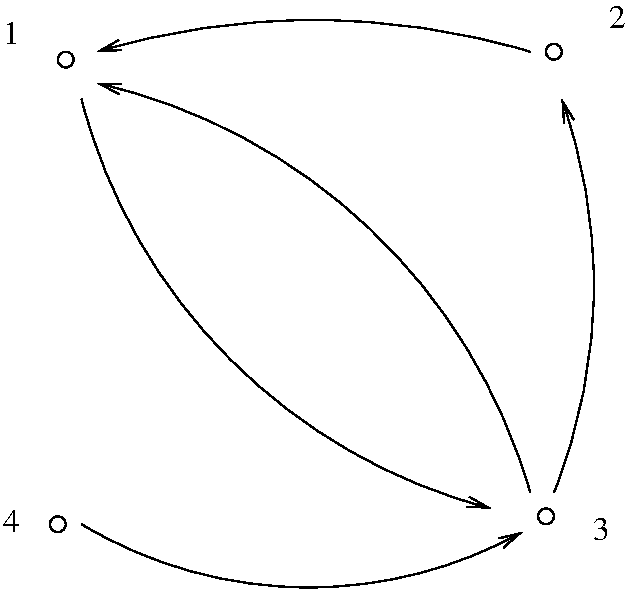
\caption{Graph for the wiring in Figure \ref{Figure1}}\label{Figure2}
\end{center}
\end{figure}
Notice that to avoid clutter we do not draw the loop from each vertex
to itself, which is always present since
a given button always switches the corresponding bulb.

\subsection{Edge function} Associated with a given directed graph $G$ is the {\it edge function}
$F:S\to 2^S$, where $j\in F(i)$ if there is an edge from $i$ to
$j$, and the {\it backward edge function} $F^{-1}:S\to 2^S$, where $j\in
F^{-1}(i)$ if there is an edge from $j$ to $i$. 
In the case where $G$ represents a 
button-bulb wiring $W$, $F(i)$ corresponds to the set of bulbs whose on/off
status changes when Button $i$ is pressed while $F^{-1}(i)$ corresponds to
the set of buttons that, when pressed, change the on/off status of Bulb $i$.
These functions 
specify the target of each outgoing edge and the
source of each incoming edge,
i.e. the head of each outgoing arrow and 
the feathers of each incoming arrow. We extend the
definitions of $F$ and $F^{-1}$ to $2^S$ in the usual way: $F(T)$ and
$F^{-1}(T)$ are the unions of $F(i)$ or $F^{-1}(i)$, respectively, over all
$i\in T\subset S$. We say that $T\subset S$ is {\it forward invariant} if
$F(T)\subset T$, or {\it backward invariant} if $F^{-1}(T)\subset T$. We
denote by $G_T$ the subgraph of $G$ consisting of the vertices in $T$ and
all edges between them.

\subsection{Matrix reformulation}
If we examine the effect of a finite sequence of vertex presses
$i_1,\dots,i_k$, on a fixed vertex $i_0$, it is clear that the final on/off
state of vertex $i_0$ depends only on its initial state and the parity of
the number of indices $j$, $1\le j\le k$, for which $i_0\in F(i_j)$. In
particular, the order of the vertices in our finite sequence is irrelevant
to the final state of $i_0$. Since this is true for each vertex, we readily
deduce the following:

\begin{itemize}
\item The order of a finite sequence of vertex presses is irrelevant to
    the final on/off states of all vertices.
\item We may as well assume that each vertex is pressed at most once,
    since pressing it twice produces the same effect as not pressing it
    at all.
\end{itemize}

Thus, instead of talking about a {\it finite sequence} of vertex presses, we
can talk about a {\it set} of vertex presses and represent this set as an
$n$-dimensional column vector $x\in\F_2^n$ (where $\F_2=\{0,1\}$ denotes the
field with two elements), with $x_i=1$ if and only if vertex $i$ is pressed
once and $x_i=0$ if it is not pressed at all. Similarly, we represent the
initial on/off state of the vertices by a column vector $c\in\F_2^n$, with
$c_i=1$ if and only if vertex $i$ is initially lit. Lastly, we represent the
wiring $W$ as an element in $\M(n,n;\F_2)$, the space of $n\times n$ matrices
over $\F_2$. To be specific, $W=(w_{i,j})$, where $w_{i,j}=1$ if and only if
vertex $j$ affects the on/off status of vertex $i$; we note that
$w_{i,i}=1$ for all $i\in S$.
The non-zero entries
in the i-th row of W lists those vertices that switch vertex i on or off. The
non-zero entries in the j-th column list those vertices that are switched on
or off by vertex j. The matrix W is, in fact, the transpose of the adjacency
matrix for the directed graph G. 
With these conventions, the vector $v=Wx+c\in\F_2^n$ is
such that $v_i=1$ if and only if vertex $i$ is lit, assuming we have initial
configuration $c$, wiring $W$, and vertex presses given by $x$. 

\subsection{Degree}
The {\it degree of vertex $i$}, $\deg(i)$, is
the number of $1$'s in the $i$th column of $W$ (or, equivalently, the
cardinality of $F(i)$. 
In graph-theoretic terms, this degree is the
\emph{out-degree} of the vertex).
 We define the {\it degree of $W$}, $\deg(W)$, to be
$\max\{\deg(i) : i\in S\}$.  

For $u\in\F_2^n$, we define $|u|$ to be the \emph{Hamming norm} or Hamming
distance from $u$ to the origin, i.e.~the number of $1$ entries in $u$. 
Then $\deg(i)$ for a wiring $W$ is the norm of the
$i$-th column of the matrix $W$.
Also, $|Wx+c|$ is the number of lit vertices, 
assuming we have initial
configuration $c$, wiring $W$, and vertex presses given by $x$. 
Thus the function $M(W,c)$ defined in the
Introduction can now be described as 
$$ M(W,c) = \max\{\,|Wx+c|: x\in \F_2^n\,\}.$$

For $n,m\ge 1$, we define $A(n,m)$ to be the set of matrices $W\in
\M(n,n;\F_2)$ that have $1$'s all along the diagonal and satisfy $\deg(W)\le
m$. If also $n\ge m$, we define $A^*(n,m)$ to be the set of matrices in
$A(n,m)$ for which $\deg(i)=m$, for all $i\in S$. These classes of matrices
are the classes of admissible wirings for the functions defined in the
Introduction:
$$
\begin{aligned}
\mu(n,m)   &= \min\{ M(W,0) : W\in A(n,m) \}\,, \\
\mu^*(n,m) &= \min\{ M(W,0) : W\in A^*(n,m) \}\,, \\
\nu(n,m)   &= \min\{ M(W,c) : W\in A(n,m),\; c\in\F_2^n \}\,, \\
\nu^*(n,m) &= \min\{ M(W,c) : W\in A^*(n,m),\; c\in\F_2^n \}\,, \\
\end{aligned}
$$
The largest class of admissible wirings on $n$ vertices that interests us is
$A(n):=A(n,n)$. This gives rise to the numbers $\mu(n):=\mu(n,n)$ and
$\nu(n):=\nu(n,n)$, as defined in the Introduction. It is convenient to
define $\mu(0,m)=0$ for all $m\in\N$.

\subsection{Connection to coding}
Although the Hamming distance is a central part of the problems under
consideration, these problems are on the surface quite different from those
in coding theory, since we are looking for wirings that minimize the maximum
distance from the origin of $Mx$, $x\in\F_2^n$, whereas in coding theory we
are looking for codes that maximize the minimum distance between codewords.
However, it is shown in \cite{BOF} that Sylvester-Hadamard matrices, which
are known to give rise to Hadamard codes that possess a certain optimality
property, also give rise to certain optimal wirings.

\subsection{Augmented complete graphs}
In graph theory, a \emph{complete directed graph on $r$ vertices}
(also called a $K_r$) has an edge from each vertex to each other vertex.
A wiring of $r$ bulbs
for which each button switches all the bulbs corresponds to 
a graph which has a $K_r$ augmented by a loop at each vertex.
We call such a graph an \emph{augmented complete graph},
or a $\hat K_r$.
Given the graph $G$ of a wiring,
we say that a subgraph $H$ of $G$ is an {\it augmented
complete
subgraph on $r$ vertices}, or \emph{a $\hat K_r$ in $G$,}
if there is an edge from every vertex of $H$ to
every vertex of $H$.  If $H$ is such a subgraph, we call
the set of its vertices  
a \emph{$\hat K_r$ set in $G$}.

For $t\in\{0,1\}$, we denote by $t_{p\times q}$ the $p\times q$ matrix all
of whose entries equal $t$, and let $t_p=t_{p\times p}$. The matrix $1_p$
should not be confused with the $p\times p$ identity matrix  $I_p$.

\subsection{Relationship to Satisfiability}\label{SS:MXS}
The problems under consideration in this article are closely related to
MAX-XOR-SAT problems in Computer Science. These problems are in the general
area of propositional satisfiability ({\it SAT}). To be specific, we want to
assign values to Booloean variables so as to maximize the number of clauses
that are true, where each clause is composed of a set of variables connected
by XORs. Since XOR in Boolean logic corresponds to addition mod 2, this
problem can be written in our notation as follows: given a matrix $W\in
\M(N,n;\F_2)$, we wish to choose a {\it variables vector}
$x=(x_1,\dots,x_n)\in\F_2^n$ so as to maximize the Hamming norm $|Wx|$;
the $N$ entries in $Wx\in\F_2^N$ are the {\it clauses}. Thus the goal is to
compute $M(W,0)$.

XOR-SAT and MAX-XOR-SAT have been studied extensively; see
for instance \cite{CD1}, \cite{CD2}, \cite{CDE}, \cite{DR}. Algorithms for
solving such problems are useful in cryptanalysis \cite{DGMMPR}, \cite{SNC}.

The relationship between MAX-XOR-SAT and our wiring problem is plain to see,
so let us instead mention the differences:
\begin{itemize}
\item MAX-XOR-SAT is concerned with finding $M(W,0)$ for a fixed but
    arbitrary $W$, rather than seeking the minimum of $M(W,0)$ over a
    class of admissible wirings $W$. The main problems in MAX-XOR-SAT revolve
    around the efficiency of the computation of $M(W,0)$ for large $n$
    rather than the computation of a minimum for all $n$.
\item In MAX-XOR-SAT, there is no requirement that $N=n$, and so no
    matching of clauses with variables (or bulbs with buttons in our
    terminology) and no requirement that $w_{ii}=1$.
\item In MAX-XOR-SAT and other SAT problems, the typical simplifying
    assumption is that there are either exactly, or at most, $m$
    variables in each clause. Thus in SAT we typically bound the Hamming
    norms of the rows of $W$, while in our wiring problem we bound the
    Hamming norms of the columns of $W$.
\end{itemize}

In spite of the differences, we would hope that the lower bounds in $M(W,0)$
given by our results might be of some interest to MAX-XOR-SAT researchers.

\section{Formulae for $\nu$ and $\nu^*$}\label{S:nu}

\subsection{Trivial bounds}
Loosely speaking, larger sets of numbers have smaller minima. More
precisely,
if $E\subset F\subset\N$, then $\min F\le \min E$. 
Thus
given $n\ge m$, the following inequalities are immediate:
\begin{align}
\nu(n,m) &\le \nu^*(n,m) \le \mu^*(n,m) \\
\nu(n,m) &\le \mu(n,m) \le \mu^*(n,m)
\end{align}

\subsection{A lower bound for $M(W,c)$}

\begin{lem}\label{L:mean}
Let $n\in\N$. For all $W\in A(n)$ and $c\in\F_2^n$, the mean value of
$|Mx+c|$ over all $x\in\F_2^n$ is $n/2$. In particular, $M(W,c)\ge n/2$ and
$M(W,c)>n/2$ if the cardinality of $\{i\in[1,n]\cap\N: c_i=1\}$ is not
$n/2$.
\end{lem}

\begin{proof}
Fix $W$ and $c$. Let $S_i=\{x\in\F_2^n: x_i=0\}$ and $T_i=\F_2^n\setminus
S_i$. Both $S_i$ and $T_i$ have cardinality $2^{n-1}$.
Then, $f:S_i \to T_i$ is a bijection, where $f (x)$ differs from $x$ in
the $i$-th position and only in the $i$-th position. Since pressing vertex $i$ toggles
its own on/off status, $(W x + c)_i = 1$ if and only if $(W f (x) + c)_i = 0$. Let $k$ be
the number of sets of vertex presses $x$ in $S_i$ for which $(W x + c)_i = 1$. Then
exactly $k$ sets of vertex presses $x$ in $T_i$ lead to $(W x + c)_i = 0$ and so 
$2^{n-1}-  k$
lead to $(W x + c)_i = 1$. In total, therefore, there are $2^{n-1}$ sets of vertex presses
$x$ in $\F_2^n$
for which $(W x + c)_i = 1$. The mean value of $(W x + c)_i$ is therefore
$\frac12$ for each $i$. The mean value of $|W x + c|$ is then $n/2$ since this mean value
is given by
$$\frac1{2^n} \sum_{ x\in \F^n_2} |W x + c| = 
\frac1{2^n} \sum_{ x\in \F_2^n} \sum_{i=1}^n (W x + c)_i =
\sum_{i=1}^n \frac1{2^n} \sum_{ x\in \F_2^n} (W x + c)_i =
\frac{n}2.$$
The last statement
in the lemma follows easily.
\end{proof}

\subsection{}
The above lemma is a key tool in proving the following result which gives
the general formula for $\nu(n,m)$ and $\nu^*(n,m)$. In this result, we
ignore the case $m=1$ since trivially $\nu(n,1)=\nu^*(n,1)=n$.

\begin{thm}\label{T:nu}
Let $n,m\in\N$, $m>1$.
\begin{enumerate}
\item $\nu(n) = \nu(n,m) = \lce \ds{\frac{n}2} \rce$.
\item If $n\ge m$, then
$$
\nu^*(n,m) =
  \begin{cases}
  \nu(n,m)+1, &\text{if $\/n$ is even and $\/m$ odd}, \\
  \nu(n,m),   &\text{otherwise}.
  \end{cases}
$$
In particular, $\nu^*(n,2)=\nu^*(n)=\nu(n)$ for all $n>1$.
\end{enumerate}
\end{thm}

\begin{proof} We will prove each identity by
showing that the right-hand side is both a lower and an upper bound
for the left-hand side.

	By \rf{L:mean}, $M(W,c)\ge\lce\ds{\frac{n}2}\rce$ for all $W\in A(n)$ and
$c\in\F_2^n$. This global lower bound yields the desired lower bound for
$\nu(n)$ and {\it a fortiori} for $\nu(n,m)$ and for $\nu^*(n,m)$ except in
the case where $n$ is even and $m$ is odd.
We postpone the proof of the lower bound in this case, until
we have completed the proof of (a).

\smallskip
To prove the reverse inequalities, the upper bounds, we take as our initial
configuration the {\it even indicator vector} $e\in\F_2^n$ defined by
$e_i=1$ when $i$ is even, and $e_i=0$ when $n$ is odd. We split the set of
integers between $1$ and $n$ into pairs $\{2k-1,2k\}$, $1\le k\le n/2$, with
$n$ being unpaired if $n$ is odd; corresponding to the pairs of integers, we
have {\it pairs of rows} in the wiring matrix $W$ and {\it pairs of
vertices}. For each proof of sharpness, we will define $W=(w_{i,j})$ such
that $M(W,e)$ equals the desired lower bound. Pressing vertex $j$ has no
effect on the pair of vertices $2k-1$ and $2k$ if $w_{2k-1,j}=w_{2k,j}=0$,
and it toggles both of them if $w_{2k-1,j}=w_{2k,j}=1$. Since initially one
vertex in each pair is lit, this remains true regardless of what vertices we
press if the corresponding pair of rows are equal to each other (as will be
the case for most pairs of rows). Thus, in calculating $M(W,e)$, we can
ignore all pairs of equal rows, for which the corresponding vertex presses
leaves the number of lit vertices unchanged, and we only have to consider the
unpaired vertex, if present.

To finish the proof of (a), it suffices to show that $\nu(n,2)\le \lce
	\ds{\frac{n}2} \rce$. Define the $n\times n$ block diagonal matrix
\begin{equation}\label{E:nu-W}
W =
  \begin{cases}
  \diag(1_2,\dots,1_2),     & n\text{ even,} \\
  \diag(1_2,\dots,1_2,1_1), & n\text{ odd,}
  \end{cases}
\end{equation}
In case $n=9$, this matrix
corresponds to the wiring of nine buttons and bulbs represented
by Figure \ref{Figure3}. In this figure, the boxes labelled by
the number
$2$ represent augmented complete directed graphs on two vertices,
	and the small circle represents a single vertex (and its loop).
We shall always indicate an augmented complete $\hat K_v$ subgraph by a box
labelled $v$.
\begin{figure}[ht]
\begin{center}
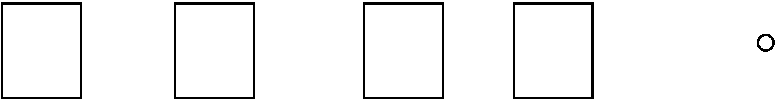
\caption{$n=9$}\label{Figure3}
\end{center}
\end{figure}

	Then $W\in A(n,2)$ and $M(W,e)=\lce \ds{\frac{n}2} \rce$. To see this, note
that rows $2k-1$ and $2k$ of $W$ are equal to each other for each $1\le k\le
n/2$. Thus when $n$ is even, $|Wx+e|$ is independent of $x$, while it
toggles between the two values $r$ and $r-1$ when $n=2r-1$ is odd, due to
the change in the state of vertex $n$ each time that vertex is pressed.

Now we prove the lower bound for (b)
in the exceptional case. 
Fix $c\in\F_2^n$ and $W\in A^*(n,m)$ for some odd $m>1$ and $n\ge m$. Each
vertex press must change the parity of the number of lit vertices and, since
the mean value of $|Wx+c|$ is $n/2$, it follows that $|Wx+c|>n/2$ for some
$x\in\F_2^n$. Since $\nu(n,m)=n/2$ if $n$ is even, we deduce that
$\nu^*(n,m)\ge\nu(n,m)+1$ if $n$ is even and $m$ odd.

It remains to prove that the desired formula in (b) for $\nu^*(n,m)$ is also an
upper bound for $\nu^*(n,m)$ when $n\ge m>1$. Suppose first that $n-m$ is
even. First, define the block diagonal matrix $W'\in A(n,m)$ by the formula
$W'=\diag(1_2,\dots,1_2,1_m)$, where there are $(n-m)/2$ copies of $1_2$. We
modify $W'=(w'_{i,j})$ to get a matrix $W=(w_{i,j})\in A^*(n,m)$ by
adding $m-2$ $1$'s to the end of each of the first $n-m$ columns, i.e.~let
$$
w_{i,j} =
  \begin{cases}
  1, &\text{$i>n-m+2$ and $j\le n-m$,} \\
  w'_{i,j}, &\text{otherwise}
  \end{cases}
$$
In case $n=9$ and $m=3$, the matrix $W$ corresponds
to a wiring of the kind indicated in Figure \ref{Figure4}.
\begin{figure}[ht]
\begin{center}
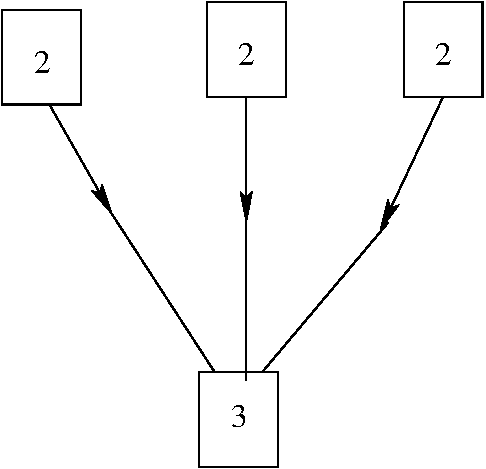
\caption{$n=9$, $m=3$}\label{Figure4}
\end{center}
\end{figure}
In this diagram, the boxes indicate augmented complete subgraphs
having two or three vertices, as indicated.
A single arrow coming from a $\hat K_2$ box indicates an edge
from {\em each} of the two vertices in the box
and going to {\em the same} vertex in the $\hat K_3$. The target
vertex may be the same or different
for the three $\hat K_2$'s, but the vertices in
a given $\hat K_2$ share the same target.  In general, in our diagrams,
we will use the convention that \textbf{all the buttons corresponding
to vertices in a given $\hat K_r$ box produce exactly the same effect.}
Notice that nonisomorphic graphs may correspond to the
same \lq\lq box diagram", in view of the fact that a box
diagram is not specific about the targets of some
arrows.

All paired rows of $W$ are equal, so if $n$ and $m$ are both even, then
$|Wx+e|=n/2$ for all $x\in\F_2^n$, whereas if $n$ and $m$ are both odd, the
value of $|Wx+e|$ is either $(n+1)/2$ or $(n-1)/2$, depending on the parity
of $|x_i|$. In either case, we have found a matrix $W\in A^*(n,m)$ with
$M(W,e)=\nu(n,m)$, and so $\nu^*(n,m)=\nu(n,m)$.

Suppose next that $n$ is odd and $m$ even, with $n>m+1$. We first define the
block diagonal matrix $W'\in A(n,m)$ by the formula
$W'=\diag(1_m,1_2,\dots,1_2,W_3)$, where there are $(n-m-3)/2$ copies of
$1_2$ and
\begin{equation}\label{E:W3}
W_3 = \begin{pmatrix} 1 & 0 & 1 \\ 1 & 1 & 0 \\ 0 & 1 & 1 \end{pmatrix}\,.
\end{equation}
and then define $W=(w_{i,j})$ by the equation
\begin{equation} \label{E:WW'}
w_{i,j} =
  \begin{cases}
  1, &\text{$3\le i\le m$ and $j>m$,} \\
  w'_{i,j}. &\text{otherwise}
  \end{cases}
\end{equation}
The corresponding wiring is indicated schematically
in Figure \ref{Figure5}.
\begin{figure}[ht]
\begin{center}
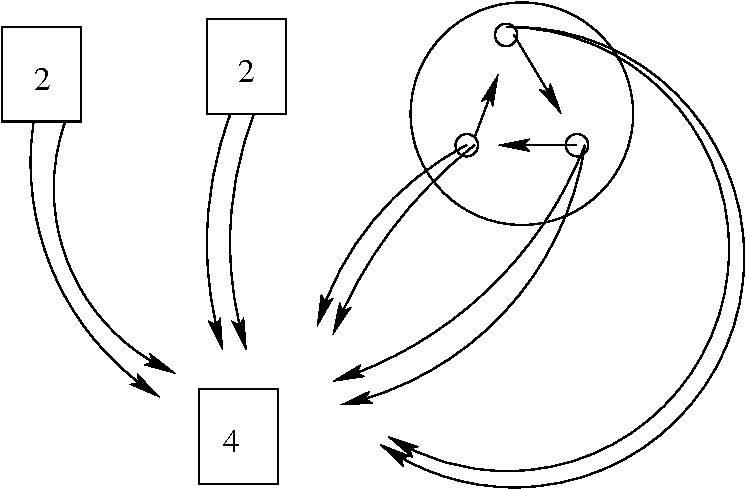
\caption{$n=11$, $m=4$}\label{Figure5}
\end{center}
\end{figure}

The circled subgraph corresponds to the matrix $W_3$. The double
arrows coming from each $\hat K_2$ each represent four edges in the graph,
i.e. two pairs of edges, where each pair has a distinct target
and the $\hat K_2$ set is the set of sources for the pair. 

The first $n-3$ rows can be split into duplicate pairs as before, so the
associated pairs of vertices will always be of opposite on/off status and
the number of them that is lit is always $(n-3)/2$.

Initially, two of the last three vertices are lit.  Since $m$ is even, the
parity of the number of lit vertices is preserved, and so no more than two
of the last three vertices can be lit. Thus, $M(W,e)=(n+1)/2$ in this case,
as required.

The case where $m$ is odd and $n>m+1$ is even, is similar. We first define
$W'\in A(n,m)$ by the formula $W'=\diag(1_m,W_3,1_2,\dots,1_2)$, and then
define $W=(w_{i,j})$ from $W'$ by \rf{E:WW'}.
The corresponding wiring is indicated schematically
in Figure \ref{Figure6}. (In this figure, following our convention,
we
indicate the multiple edges emanating from a $\hat K_2$ and going to
the same target node by
a single edge.)
\begin{figure}[ht]
\begin{center}
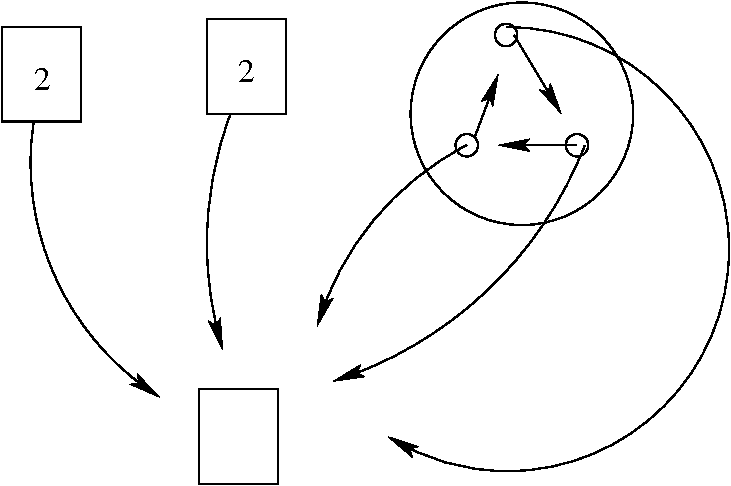
\caption{$n=10$, $m=3$}\label{Figure6}
\end{center}
\end{figure}

There are four unpaired rows,
namely rows $i$, $m\le i\le m+3$. By an analysis similar to the previous
case, at most three of these vertices can be lit (namely vertex $m$ and at
most two of the other three vertices), and half of the remaining $n-4$
vertices are always lit. It follows that $M(W,e)=(n+2)/2$, as required.

Finally, if $n=m+1$, we define $W$ to be the block diagonal matrix
$$
W =
  \begin{pmatrix}
  1_{(m-1)\times m}\hfill & 1_{(m-1)\times 1}\hfill \\
  1_{1\times m}\hfill & 0_{1\times 1}\hfill \\
  0_{1\times m}\hfill & 1_{1\times 1}\hfill
  \end{pmatrix}
$$
See Figure \ref{Figure7}.
\begin{figure}[ht]
\begin{center}
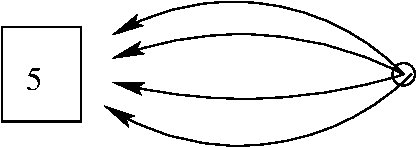
\caption{$n=6$, $m=5$}\label{Figure7}
\end{center}
\end{figure}

The first $m-2$ or $m-1$ rows are paired, depending on whether $m$ is even or
odd, respectively. Thus, $M(W,e)\le 1+m/2$ if $m$ is even, or $M(W,e)\le
2+(m-1)/2$ if $m$ is odd, as required.
\end{proof}

\subsection{Sublinearity}
Generalizing an idea used in the above proof, we see that if $W$ and $c$
have block forms
$$
W = \begin{pmatrix} W_a & 0 \\ 0 & W_b \end{pmatrix}\,\qquad
c = \begin{pmatrix} c_a \\ c_b \end{pmatrix}\,,
$$
then
\begin{equation}\label{E:WaWb}
M(W,c)=M(W_a,c_a)+M(W_b,c_b)\,.
\end{equation}
This readily yields the following:
\begin{cor}
If $\la$ is any one of the four functions $\mu$,
$\mu^*$, $\nu$, or $\nu^*$, then it is sublinear in the first variable:
\begin{equation}\label{E:sublinear}
\la(n_1+n_2,m) \le \la(n_1,m) + \la(n_2,m)\,,
\end{equation}
as long as this equation makes sense (i.e.~we need $n_1,n_2\ge m$ if
$\la=\mu^*$ or $\la=\nu^*$).
\end{cor}

\section{The case $m=2$} \label{S:m=2}

\subsection{Proof of \rf{T:m=2}}
\begin{proof}
Trivially $\mu(1,2)=1$, and it is easy to check that $\mu(2,2)=2$. Taking
$W_3$ as in \rf{E:W3}, we see that $M(W_3,0)=2$, and so
$\mu(3,2)\le\mu^*(3,2)\le 2$. By combining \rf{E:sublinear} with these
facts, we see that for $k\in\Z$, $k\ge 0$, and $i\in\{0,1,2\}$,
$$ \mu(3k+i,2) \le k\mu(3,2)+\mu(i,2)\le 2k+i\,. $$
	Since $2k+i=\lce\ds{\frac{2(3k+i)}{3}}\rce$, this gives the sharp upper bound
for $\mu(n,2)$. The corresponding sharp upper bound for $\mu^*(n,2)$ follows
similarly when $n\ge 1$ has the form $3k$ or $3k+2$, $k\ge 0$. If $n=3k+1$,
$k\ge 1$, only a small change is required to the $\mu$-proof to get a proof
of the sharp $\mu^*$ upper bound:
$$ \mu^*(3k+1,2) \le (k-1)\mu^*(3,2)+2\mu^*(2,2) = 2k+2\,. $$

It remains to show that we can reverse the above inequalities. We first
examine the reverse inequalities for $\mu^*$, so fix $W\in A^*(n,2)$.
	Writing $F:S\to 2^S$ for the edge function, where $S:=S(n)$, we get
a well-defined function $f:S\to S$ by writing $f(i)=j$ whenever there is an
edge from $i$ to $j\not=i$ in the associated graph $G$. For a dynamical system on
any finite set, every point is either periodic or preperiodic. In our
context, this just means that if we apply $f$ repeatedly starting from any initial
vertex $i\in S$, then we eventually get a repeat of an earlier value, and
from then on the iterated values of $f$ go in a cycle.

Note that the topological components of $G$ do not ``interfere'' with each
other: the vertices in any one component affect only the on/off status of
vertices in this component, so maximizing the number of lit vertices can be
done one component at a time (alternatively, this follows from \rf{E:WaWb}
after reordering of the vertices).

A component of the graph $G$ consists of a central circuit containing two or
more vertices, with perhaps some directed trees, each of which leads to some
vertex of the circuit, which we call the {\it root} of that tree. Starting
from the outermost vertices of such a tree (those that are not in the range
of $f$) and working our way down to the root, it is not hard to see that we
can simultaneously light all vertices in each of these trees. Having done
this, some of the vertices in the central circuit may not be lit up. We
follow the vertices around the circuit in cyclic order, pressing each vertex
that is unlit when we reach it until we have gone fully around the circuit.
It is clear that at this stage at most one vertex in the circuit is unlit,
and all the associated trees (excluding the roots) are still fully lit.
\begin{figure}[ht]
\begin{center}
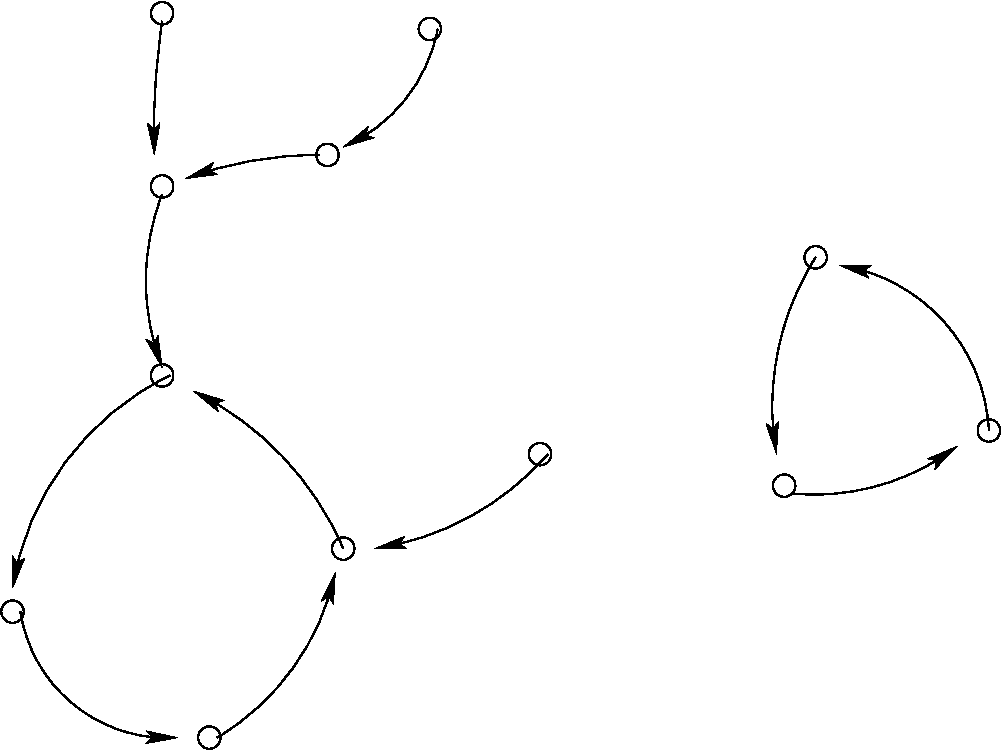
\caption{\lq Dynamics' of $m=2$}\label{Figure8}
\end{center}
\end{figure}

Note that any single vertex press either leaves the number of lit vertices
in a given component unchanged, or changes that number by $2$. Since
initially all vertices are unlit, it follows that the number of lit vertices
in a component is always even. It therefore follows that in a component of
even cardinality all vertices can be lit, while in a component of odd
cardinality all except one can be lit.

Thus, it follows that to minimize $M(W,0)$ we need to maximize the number of
components of odd cardinality (necessarily at least 3), and that the maximum
proportion of lit vertices in any one component is at least $2/3$ (with
equality only for components of cardinality $3$). Thus $\mu^*(n,2)\ge \lce
2n/3\rce$, which gives the required lower bound except when $n=3k+1$,
$k\in\N$. Since $G$ has $n=3k+1$ vertices and all components have at least
two vertices, it can have at most $k-1$ components of odd cardinality,
yielding the desired estimate $\mu^*(3k+1,2)\ge 3k+1-(k-1)=2k+2$. Thus
$\mu^*(n,2)$ is given by the stated formula in all cases.

For $\mu$, the above proof goes through with little change. We define $f(i)$
as before whenever Button $i$ switches two bulbs, and $f(i)=i$ otherwise.
The graph $G$ can now have {\it prefixed components} where the central
circuit contains only a single vertex, corresponding to a fixed point of
$f$. However, it is clear from our earlier arguments that prefixed
components can always be fully lit, so only the odd cardinality
{\it non-prefixed components} (i.e.~those without a fixed point) can
contribute unlit bulbs. Thus, $\mu(n,2)\ge\lce 2n/3\rce$, as required.
\end{proof}

Although prefixed components do not contribute unlit vertices in the last
paragraph of the above proof, singleton components (corresponding to a
vertex with no inbound or outbound edge) are important since they
allow us to get $k$, rather than just $k-1$, non-prefixed components of odd
cardinality when $n=3k+1$. This accounts for the difference between
$\mu(n,2)$ and $\mu^*(n,2)$ in this case.

It follows from the above proof that a wiring minimizes $M(W,0)$ in either
$A^*(n,2)$ or $A(n,2)$ if and only if its associated graph maximizes the
number of non-prefixed components of odd cardinality among the allowed set
of graphs. Such components have cardinality at least $3$ so, for $n$
a multiple of $3$, this means that each component must have three vertices
and correspond (up to permutation) to one or other of the matrices
$$
\left(\begin{matrix} 1&0&0\\1&1&1\\0&1&1 \end{matrix}\right)
\textup{ or }
\left(\begin{matrix} 1&0&1\\1&1&0\\0&1&1 \end{matrix}\right).
$$
For $n$ of the form $3k+1$ or $3k+2$, it similarly follows easily from the
extremality criterion that all components except at at most two are of
cardinality $3$ and have one of the two above forms. The possible
exceptional components depend on the mod-$3$ nature of $n$, as well as
whether we are looking at $A^*(n,2)$ or $A(n,2)$, but all are of
cardinality $2$, $4$, $5$, or $7$. We leave to the reader the routine but
tedious task of using the above extremality criterion to find all such
sets of exceptional components.

\section{Pivoting and the case $m=3$} \label{S:m=3}
\subsection{Pivoting}
In preparation for the proof of \rf{T:m=3}, we introduce the concept of {\it
pivoting}. Pivoting about a vertex $i$, $1\le i\le n$, is a way of changing
the given wiring $W$ to a special wiring $W^i$ such that $M(W^i,c)\le
M(W,c)$. Additionally, pivoting preserves the classes $A(n,m)$ and
$A^*(n,m)$. 
\begin{figure}[h]
\begin{center}
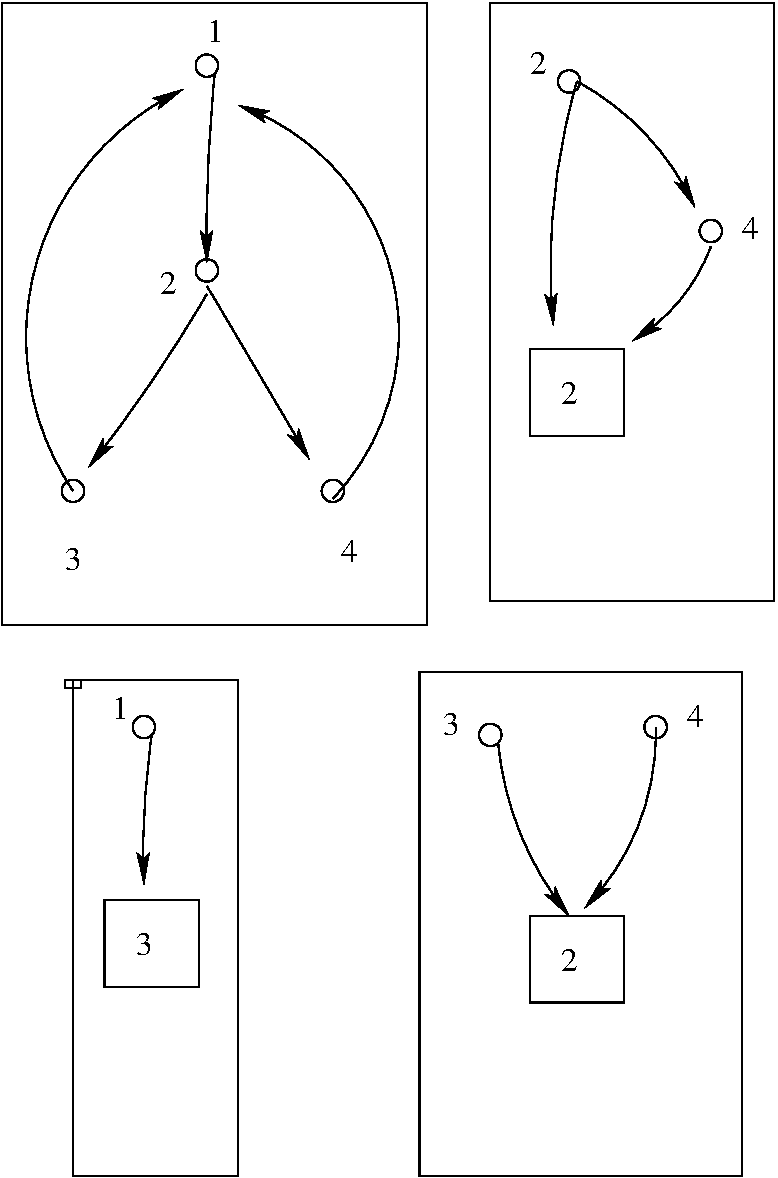
\caption{Pivoting}\label{Figure9}
\end{center}
\end{figure}

Fix a wiring $W=(w_{i,j})$ and initial configuration $c$, and let $F:S\to
2^S$ denote the edge function associated to $W$, where $S=S(n)$.
Given $i\in S$, let $M_i=M(W^i,c)$ where the {\it pivoted wiring matrix}
$W^i$ is defined by the condition that its $j$th column equals the $i$th
column of $W$ if $j\in F(i)$, and equals the $j$th column of $W$ otherwise.
In other words, $W^i$ rewires the system so that pressing the $j$th vertex
has the same effect as pressing the $i$th vertex in the original system
whenever $j\in F(i)$. On the other hand, it is easy to see that $M_i$ is the
maximum value of $|Wx+c|$ over all vectors $x$ such that $x_j=0$ whenever
$j\in F(i)\setminus\{i\}$. 
In fact, any attainable set of lit bulbs for
the wiring $W^i$ and initial configuration $c$ can be achieved without pressing any
of the buttons in $F(i)\setminus\{i\}$. Hence, the same set of lit bulbs can be achieved
with the original wiring $W$ without pressing any of those buttons.
In particular, $M_i\le M(W,c)$.
See Figure \ref{Figure9} for examples.

Pivoting about $i$, as defined above, is a process with several nice
properties:%
\begin{itemize}
\item it does not increase the value of $M$: $M(W^i,c)\le M(W,c)$;%
\item it preserves membership of the classes $A(n,m)$ and $A^*(n,m)$;%
(In fact, if $j\in F(i)$,  then $W^i_{j,j} = 1$, that is 
$W^i$ still has $1$’s along
the diagonal. This is the only property that actually requires checking in order
to verify that the classes $A(n, m)$ and $A^*(n, m)$ are preserved.)
\item if $F^i$ is the edge function of $W^i$, then $F^i(i)=F(i)$ is a
    forward invariant augmented complete subgraph of the associated graph $G^i$.
\end{itemize}

It is sometimes useful to pivot {\em partially} about $i$: given $T\subset S$, and
$i\in S$, we define $W'$ by replacing the $j$th column of $W$ by its $i$th
column whenever $j\in F(i)\setminus T$. Such {\it pivoting about $i$ with
respect to $T$} satisfies the same non-increasing property, preserves
membership in $A(n,m)$ and $A^*(n,m)$, and $F(i)\setminus T$ is a (not
necessarily forward invariant) augmented complete subgraph of the associated graph
$G'$.

\subsection{}
Pivoting is the key trick in the proof of the following lemma.

\begin{lem}\label{L:alternative}
Let $m\ge 2$ and $n\ge 1$. Then either $\mu(n+m,m)=\mu(n+m,m-1)$, or
$$ \mu(n+m,m)\ge \mu(n,m)+\nu(m,m) = \mu(n,m) + \lce {m/2} \rce\,. $$
\end{lem}

\begin{proof}
Suppose $\mu(n+m,m)<\mu(n+m,m-1)$, and let $W\in A(n+m,m)$ be such that
$M(W,0)=\mu(n+m,m)$. Then, $W$ has a vertex $i$ of degree $m$. By minimality
of $W$, pivoting about $i$ gives $W^i\in A(n+m,m)$ with $M(W^i,0) =
\mu(n+m,m)$ (cf. Figure \ref{Figure10}. The loop marked
$n$ just indicates an unspecified subgraph
of order $n$.). For the wiring $W^i$, we first press a set of vertices in
\begin{figure}[ht]
\begin{center}
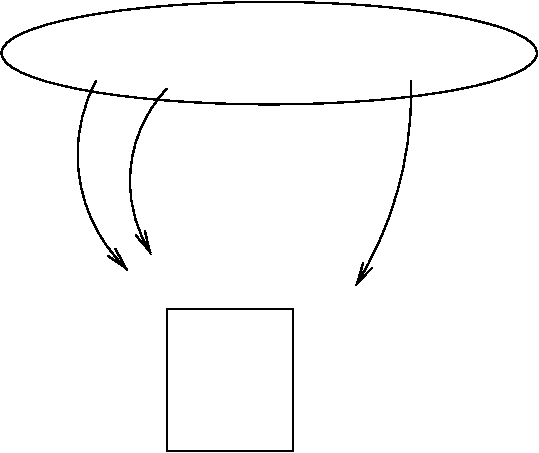
\caption{$W^i$}\label{Figure10}
\end{center}
\end{figure}
\noindent
	$S(n+m)\setminus F(i)$ so as to maximize the number of lit vertices in
	$S(n+m)\setminus F(i)$, and then we press vertex $i$ if fewer than half of the
vertices in $F(i)$ are lit. By forward invariance of $F(i)$, the result
follows.
\end{proof}

\begin{proof}[Proof of \rf{T:m=3}(a)]
Trivially, we have that $\mu(n,3)\le\mu(n,2)$, with equality if $n<3$. It is
also immediate that $\mu(3,3)=\mu(3,2)=2$: any wiring that includes a
vertex of degree $3$ allows us to light all vertices by pressing the degree
$3$ vertex.

Suppose therefore that $\mu(n',3)=\mu(n',2)$ for all $n'<n$, where $n>3$.
Either this equation still holds when $n'$ is replaced by $n$, or
$$
\mu(n,2) = \mu(n-3,2)+2 = \mu(n-3,3)+2 = \mu(n-3,3)+\nu(3,3) \le
\mu(n,3)\le \mu(n,2).
$$
Here, the first equality follows from \rf{T:m=2}, the second from the
inductive hypothesis, and the first inequality from \rf{L:alternative}.
Since $\mu(n,2)$ is at both ends of this line, we must have $\mu(n,3) =
\mu(n,2)$, and the inductive step is complete.
\end{proof}

\subsection{}
For the proof of \rf{T:m=3}(b), we need another lemma.

\begin{lem}\label{L:nn'm}
Let $n,m,n'\in\N$, with $n\ge m$. Then
$$ \mu^*(n+n',m+1)\le \mu^*(n,m)+n'\,. $$
\end{lem}

\begin{proof}
It suffices to prove the lemma subject to the restriction $n'\le n$, since
this case, the trivial estimate $\mu^*(n,m)\le n$, and sublinearity
\rf{E:sublinear}
together imply the general case. Let us therefore assume that $n'\le n$.

\ignore{
Detail:
For $n'<n$, we have
\begin{align*} 
&\mu^*(n+n',m+1)
\\
\le& 
\mu^*(n+1,m+1)+\mu^*(n'-1,m+1)
\\
\le& \mu^*(n,m)+1 + n'-1
\\
=&
\mu^*(n,m)+n'.
\end{align*}
}

Let $V=(v_{i,j})\in A^*(n,m)$ be such that $M(V,0)=\mu^*(n,m)$. We now
define a matrix $W=(w_{i,j})\in A^*(n+n',m+1)$. First the upper left block
of $W$ is a copy of $V$, i.e.~we let $w_{i,j}=v_{i,j}$ for all $1\le i,j\le
n$. Next, the $n'\times n$ block of $W$ below $V$ consists of copies of the
$n'\times n'$ identity matrix; the last of these copies will be missing some
columns unless $n$ is a multiple of $n'$. Lastly, we define
$w_{i,n+j}=w_{i,j}$ for all $1\le j\le n'$. It is straightforward to verify
that $W\in A^*(n+n',m+1)$; note that the assumption $n'\le n$
ensures that $W$ has $1$'s along the diagonal. Refer to Figure \ref{Figure11}
for a schematic. Note that vertex $6+i$ has the same
targets as vertex $i$, but these edges going to
vertices other than $7$, $8$ or $9$ are not shown.
\begin{figure}[ht]
\begin{center}
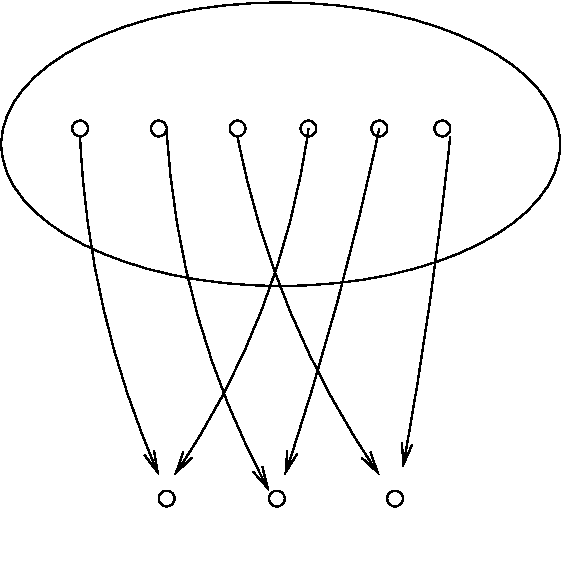
\caption{$n=6$, $n'=3$}\label{Figure11}
\end{center}
\end{figure}

Since all columns after the $n$th column are repeats of earlier columns, it
suffices to consider what happens when we press only combinations of the
first $n$ vertices. Such combinations light at most $\mu^*(n,m)$ of the
first $n$ vertices, so we are done.
\end{proof}

\subsection{Proof of \rf{T:m=3}(b)}
\begin{proof}
\rf{L:nn'm} ensures that if $k,i\in\N$, then $\mu^*(3k+i,3) \le
\mu^*(3k,2)+i = 2k+i$. This is the required sharp upper bound if $i=1,2$,
since $2k+1=\mu(3k+i,3)$ in this case. On the other hand,
$\mu^*(3k+i,3)\ge\mu(3k+i,3)=2k+i$, for all $k\in\N$ and $i=1,2$, 
and this gives the
required converse for $i=1,2$.

It remains to handle the case where $n$ is a multiple of $3$. First, we show
that the lower bound $\mu^*(3k,3)\ge\mu(3k,3)=2k$ is sharp when $k=2k'$ is
even. Letting
\begin{equation} \label{E:W6}
W_6 =
  \begin{pmatrix}
  1&0&0&0&0&0 \\ 1&1&1&0&0&0 \\ 0&1&1&0&0&0 \\
  0&1&1&1&1&1 \\ 1&0&0&1&1&1 \\ 0&0&0&1&1&1
  \end{pmatrix}
  \in A^*(6,3)\,,
\end{equation}
we claim that $M(W_6,0)=4$. Assuming this claim, \rf{E:sublinear} gives the
desired sharpness: $\mu^*(6k',3)\le k'\mu^*(6,3)\le k'M(W_6,0)=4k'$.
\begin{figure}[h]
\begin{center}
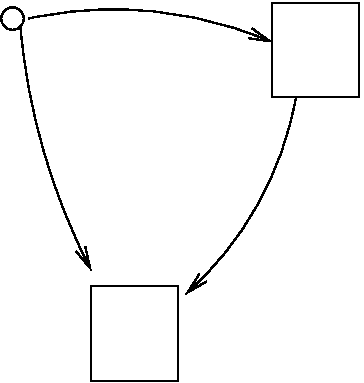
\caption{$W_6$}\label{Figure12}
\end{center}
\end{figure}

To establish the claim, it suffices to consider sets of vertex presses
involving only vertices $1$, $2$, and $4$. With this restriction, we proceed
to list all eight possible values of $x$, and deduce that $M(W_6,0)=4$:

{\renewcommand{\arraystretch}{1.1}\renewcommand{\tabcolsep}{5mm}
\begin{center}
\begin{tabular}{c|c|c}
$x^t$ & $(W_6x)^t$ & $|W_6x|$ \\ \hline
(0,0,0,0,0,0) & (0,0,0,0,0,0) & 0 \\ \hline
(1,0,0,0,0,0) & (1,1,0,0,1,0) & 3 \\ \hline
(0,1,0,0,0,0) & (0,1,1,1,0,0) & 3 \\ \hline
(1,1,0,0,0,0) & (1,0,1,1,1,0) & 4 \\ \hline
(0,0,0,1,0,0) & (0,0,0,1,1,1) & 3 \\ \hline
(1,0,0,1,0,0) & (1,1,0,1,0,1) & 4 \\ \hline
(0,1,0,1,0,0) & (0,1,1,0,1,1) & 4 \\ \hline
(1,1,0,1,0,0) & (1,0,1,0,0,1) & 3 \\ \hline
\end{tabular}
\end{center}
}
	(Here $x^t$ denotes the row-vector transpose of the column vector $x$.)

It remains to handle the case where $n=6k'-3$ for some $k'\in\N$. It is
trivial that $\mu^*(3,3)=3$. Next note that \rf{L:nn'm} ensures that for
$k\ge 2$, $\mu^*(3k,3)\le\mu^*(3k-3,2)+3=2k+1$, so we need to show that this
is sharp if $k>1$ is odd.

Supposing $\mu^*(n,3)\le2k$ for some fixed $n=3k$, $k\in\N$, $k>1$, we will
prove that $k$ must be even. Let $W=(w_{i,j})\in A^*(n,3)$ be such that
$M(W,0)\le2k$, let $S=S(n)$, and let $F:S\to 2^S$ be the edge
function associated to $W$.

We can assume that $W$ is additionally chosen so that the associated graph
$G$ has a maximal number of (disjoint) $\hat K_3$'s among all matrices $W'\in
A^*(n,3)$ for which $M(W',0)=2n/3$. The maximum number of $\hat K_3$'s is always
positive since we can get a $\hat K_3$ by pivoting about any one vertex; $\hat K_3$
sets are pairwise disjoint and forward invariant, since each vertex in a
$\hat K_3$ uses up its two allowed outbound edges within the same $\hat K_3$.

We define $A$ to be the union of all the $\hat K_3$ sets. 
If $i\in S\setminus A$,
then $F(i)\cap A$ must be nonempty, since otherwise pivoting about $i$ would
create an extra $\hat K_3$. Thus, each $i\in S\setminus A$ has at most one edge
from it to another vertex in $S\setminus A$. Suppose there is such a vertex
$i$ with $F(i)\setminus\{i\}$ not a subset of $A$. Then, we can pivot about $i$ relative to
$A$ to get a $\hat K_2$, and the only edges coming from this $\hat K_2$ are single
edges from both of its vertices to the same element in $A$. We repeat such
pivoting of vertices relative to $A$ to create more such $\hat K_2$s until this
is no longer possible. From now on, $W$ will denote this modified wiring
matrix. We denote by $B$ the union of the $\hat K_2$ vertices and write
$C=S\setminus (A\cup B)$, and we refer to each vertex in $C$ as a $\hat K_1$
(which it is, trivially).

We already know that there is an edge from each vertex in $C$ to some vertex
in $A$. If there is only a single edge from some $i\in C$ to $A\cup B$, then
there must be an edge from $i$ to some $j\in C$. Pivoting about $i$ relative
to $A\cup B$ (or equivalently, relative to $A$), we create a new $\hat K_2$,
contradicting the fact that this cannot be done. Thus, there are two edges
from each $i\in C$ to $A\cup B$. See Figure \ref{Figure13}.

\begin{figure}[ht]
\begin{center}
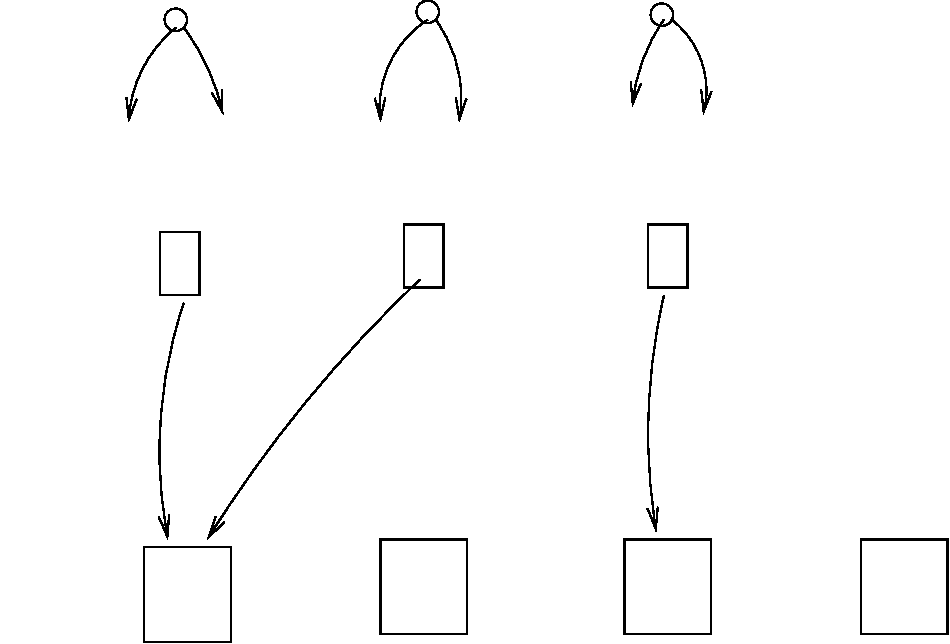
\caption{}\label{Figure13}
\end{center}
\end{figure}

We have shown that there are edges from $C$ to $A\cup B$, and from $B$ to
$A$, but that both $A$ and $A\cup B$ are forward invariant. Also, there are
no links between elements in $C$, or between elements in distinct $\hat K_2$'s or
in distinct $\hat K_3$'s. There are $3s$ elements in $A$, $2t$ elements in $B$,
and $u$ in $C$, for some integers $s,t,u$, and we have $3s+2t+u=n$.

The forward invariance of both $A$ and $A\cup B$ suggests two algorithms for
lighting many of the vertices. The first is to begin by pressing all these
vertices in $C$ to light all these vertices. After this first step, we can
ensure that at least one vertex in each $\hat K_2$ is lit by pressing a vertex in
any $\hat K_2$ without a lit vertex. Finally, we ensure that at least two
vertices are lit in each $\hat K_3$ by pressing a vertex in any $\hat K_3$ in which
fewer than two vertices are lit. Having done this, we have at least $2s+t+u$
lit vertices, so $2s+t+u\le\mu^*(n,3)$. Thus, $6s+3t+3u\le 3\mu^*(n,3)\le2n$.
When we compare this with the equation $6s+4t+2u=2n$, we deduce that $t\ge
u$.

An alternative algorithm for lighting the vertices is to first press one
vertex in each $\hat K_2$, thus lighting all $\hat K_2$ vertices. As a second step,
press a vertex in any $\hat K_3$ in which fewer than $2$ vertices are lit. Having
done this, at least two vertices in each $\hat K_3$ are lit as well as both
vertices in each $\hat K_2$. Consequently, $2s+2t\le\mu^*(n,3)\le2n/3$. Thus,
$6s+6t\le 2n$, while $6s+4t+2u=2n$. It follows that $u\ge t$, and so $u=t$.

Note that the first lighting algorithm gives at least $2s+2t=2n/3$ lit
vertices, and it actually gives more than this number unless after the first
step exactly one vertex in each $\hat K_2$ is lit. Since any larger number
contradicts $\mu^*(n,3)=2n/3$, there must be an edge from $C$ to each $\hat K_2$.
But, since the numbers of $\hat K_1$'s and of $\hat K_2$'s are equal, and there is at
most one edge from each $\hat K_1$ to $B$ (since at least one edge from each
$\hat K_1$ goes to $A$), it follows that from each $\hat K_1$ there is an edge to a
$\hat K_2$, and no other vertex in $C$ is linked to the same $\hat K_2$, i.e.~we can
pair off each $\hat K_1$ with the unique $\hat K_2$ to which it is linked in the
graph. See Figure \ref{Figure14}.
\begin{figure}[ht]
\begin{center}
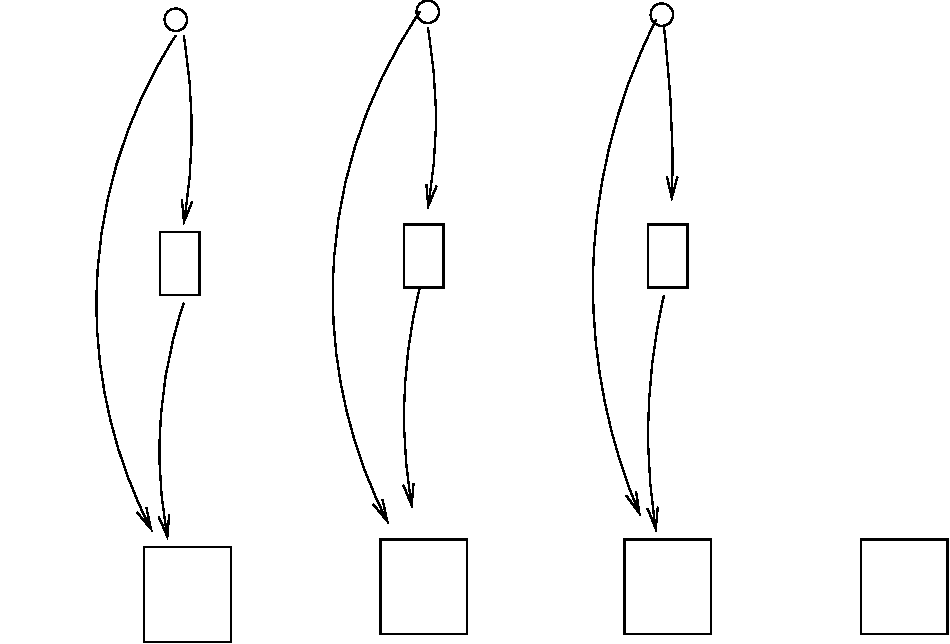
\caption{}\label{Figure14}
\end{center}
\end{figure}
We refer to the subgraph of $G$ given by the union of a $\hat K_1$ and a
$\hat K_2$ plus the edge between them as a $C_{1,2}$; the set of its three
vertices is a $C_{1,2}$ set.

The second lighting algorithm will give more than $2s+2t=2n/3$ lit vertices
unless the first step ends with one or two lit vertices in each $\hat K_3$. Thus,
there is an edge from at least one $\hat K_2$ to each $\hat K_3$. Since any one $\hat K_2$
is linked to only a single $\hat K_3$, it follows that $t\ge s$.

We now define the {\it active vertices} to be 
all $\hat K_1$ vertices, together with one vertex from
each $\hat K_2$, and the {\it active edges} are all the edges coming from active
vertices. When considering the effect of pressing sets of vertices in $B\cup
C$, we can restrict ourselves to considering only sets of active vertices,
hence the terminology.

To light more than two thirds of the vertices, it suffices to first light
two vertices in every $C_{1,2}$ set in such a way that there is at least one
$\hat K_3$ that is either fully lit or fully unlit, since we can subsequently
light two thirds of all vertices in all other $\hat K_3$ sets, together with all
vertices in the fully unlit or fully lit $\hat K_3$, by pressing only $\hat K_3$
vertices. Since each $\hat K_3$ is forward invariant, we are done.

But, given a $C_{1,2}$ set with all vertices unlit, pressing one or both of
its active vertices leaves exactly two of its vertices lit. This gives us
three ways of lighting two thirds of the vertices in that $C_{1,2}$ set, and
this flexibility will be crucial to proving that $n$ must be a multiple of
$6$. In particular, it means that for any given $\hat K_3$, there must be an
associated $C_{1,2}$ both of whose active vertices have edges leading to
that $\hat K_3$, since if this were not so, we could light two vertices in each
$C_{1,2}$ without ever pressing a vertex linked to that $\hat K_3$. Furthermore,
even if a $C_{1,2}$ is doubly linked to a $\hat K_3$, but the two active edges
between them connect to the same vertex, then by pressing both active
vertices, the on/off status of all vertices in the $\hat K_3$ remains unchanged.
Let us therefore say that a $C_{1,2}$ set with two active links to distinct
vertices in a $\hat K_3$ is {\it well linked} to that $\hat K_3$ set. We say that they
are {\it badly linked} if they are linked but not well linked.

It follows that $S$ can be decomposed into a collection of $C_{1,2}$ sets,
each of which is paired off with a distinct $\hat K_3$ set to which it is well
linked, plus $t-s$ extra $C_{1,2}$ sets that have not been paired off with
any $\hat K_3$, but are linked (well or badly) to some of the $\hat K_3$'s. We claim
that if $t>s$ then the residual $C_{1,2}$ sets always allow us to arrange
that at least one $\hat K_3$ is fully lit or fully unlit after we light two
vertices in every $C_{1,2}$. It follows from this claim 
that $n$ cannot be an odd
multiple of $3$, since then we would have $t-s>0$, and we could light more
than two thirds of the vertices.

Suppose therefore that $t>s$, and so there exists some particular $\hat K_3$ with
vertex set $D=\{a,b,c\}$, say, that has more than one $C_{1,2}$ linked to
it, at least one of which is well linked.  We wish to show that we can press
one or both of the active vertices in each of the $C_{1,2}$'s linked to $D$
while keeping $D$ {\it in sync} (meaning that all three of its
vertices are in the same on/off state).

Now $D$ is initially in sync, and we can handle any two well-linked
$C_{1,2}$'s while keeping $D$ in sync. To see this, note that if the two
pairs of active links go to the same pair of vertices in $D$, then we press
all four active vertices in both $C_{1,2}$'s. If on the other hand, they do
not go to the same pair of vertices then without loss of generality, one
$C_{1,2}$ is linked to $a$ and $b$ and the other to $b$ and $c$. By pressing
three of the four active vertices, we can toggle the on/off status of all
three vertices in $D$.

Since we can handle well-linked $C_{1,2}$'s two at a time, and we can handle
badly linked ones one at a time, while keeping $D$ in sync, we can reduce to
the situation of having to handle only two or three $C_{1,2}$'s, with at
least one of them well linked. We have already handled the case of two
well-linked $C_{1,2}$'s, so assume that there are two $C_{1,2}$'s and
exactly one is well linked, to $a$ and $b$, say, while the other is badly
linked, with either one or two links to a single vertex $v\in D$. By
symmetry, we reduce to either of two subcases: if $v=a$, then we press one
active vertex in both $C_{1,2}$'s that is connected to $a$, while if $v=c$,
then we press three vertices so as to toggle the on/off status of all of
$D$.

There remains the case of three linked $C_{1,2}$'s. If two are well linked
and one badly linked, then we just handle the two well-linked ones together
as above, and separately handle the badly linked one. Finally, all three may
be well linked. If all three $C_{1,2}$'s link to the same pair of vertices,
$a$ and $b$, say, then we press both active vertices in one of them and one
in the other two, to ensure that both $a$ and $b$ are toggled twice (and so
unchanged). If two $C_{1,2}$'s link to the same pair of vertices, $a$ and
$b$, say, and the third links to $b$ and $c$, say, then we can press one
vertex in each $C_{1,2}$ to ensure that all three vertices in $D$ are
toggled once. Finally, if no two $C_{1,2}$'s leads to the same pair of
vertices, then one leads to $a,b$, another to $b,c$, and a third to $c,a$.
We can press all six of the active vertices so as to toggle each of $a,b,c$
twice. This finishes the proof of the theorem.
\end{proof}

\subsection{Remark}
Note that even when $n$ is a multiple of $6$, the above argument gives us
some extra information: after suitable pivoting, any wiring $W\in A^*(n,3)$
with $M(W,0)=2n/3$ must reduce to a collection of $C_{1,2}$'s each of which
is well linked to a distinct $\hat K_3$. Each associated subgraph with six
vertices is a component of the full graph and is unique (up to relabeling of
the vertices). Moreover, it is the graph of the wiring $W_6$ in \rf{E:W6} so,
after suitable pivoting, any wiring $W\in A^*(n,3)$ with $M(W,0) = 2n/3$
reduces to $n/6$ copies of $W_6$.


\bigskip
\noindent
{\small
\authorbio
}

\end{document}